\documentclass[12pt]{amsart}
\usepackage{amsmath,amssymb,amsthm,graphicx,setspace,verbatim,fullpage,bbm,nicefrac, hyperref}
\usepackage[usenames]{xcolor}
\newtheorem{theorem}{Theorem}
\newtheorem{lemma}{Lemma}
\newtheorem{question}{Question}

\newtheorem{remark}{Remark}

\theoremstyle{definition}

\newcommand{\e}[2][]{\mathbb{E}_{#1}\left[#2\right]}
\newcommand{\p}{{\mathbb{P}}}

\DeclareMathOperator{\vol}{vol}

\newcommand{\eps}{\epsilon}
\newcommand{\ba}{\backslash}

\newcommand{\N}{\ensuremath{\mathbb{N}}}
\newcommand{\R}{\ensuremath{\mathbb{R}}}

\makeatletter
\def\Ddots{\mathinner{\mkern1mu\raise\p@
\vbox{\kern7\p@\hbox{.}}\mkern2mu
\raise4\p@\hbox{.}\mkern2mu\raise7\p@\hbox{.}\mkern1mu}}
\makeatother

\newcommand{\paren}[1]{\ensuremath{\left( #1 \right)}}

\newcommand{\ceil}[1]{\ensuremath{\left\lceil #1 \right\rceil}}

\newcommand{\diam}[1]{\hbox{diam}(#1)}

\newcommand{\bigTheta}[1]{\ensuremath{\Theta\!\paren{#1}}}

\newcommand{\lilOh}[1]{\ensuremath{\mathit{o}\!\paren{#1}}}
\newcommand{\bigOh}[1]{\ensuremath{\mathcal{O}\!\left( #1 \right)}}
\newcommand{\bigOmega}[1]{\ensuremath{\Omega\!\left( #1\right)}}

\begin{document}

\title{On expansion of $G_{n, d}$ with respect to $G_{m, d}$}
\author{Ioana Dumitriu}
\address{University of Washington, Department of Mathematics}
\email{dumitriu@math.uw.edu}
\author{Mary Radcliffe}

\email{radcliffe@math.uw.edu}

\begin{abstract}
In several works, Mendel and Naor have introduced and developed theory surrounding a nonlinear expansion constant similar to the spectral gap for sequences of graphs, in which one considers embeddings of a graph $G$ into a metric space $X$ \cite{mendel2010towards, mendel2013nonlinear, mendel2014expanders}. Here, we investigate the open question of whether the random regular graph $G_{n, d}$ is an expander when embedded into the metric space of a random regular graph $G_{m, d}$  a.a.s., where $m\leq n$. We show that if $m$ is fixed, the answer is affirmative. In addition, when $m\to \infty$, we provide partial solutions to the problem in the case that $d$ is fixed or that $d\to \infty$ under the constraint $d=\lilOh{m^{1/2}}$.
\end{abstract}
\maketitle

\section{Introduction}
Since the 1970s, expander graphs have been a subject of interest in both mathematics and computer science. Applications for expanders have been found in algorithms, error correcting codes, network analysis, group theory, geometry, topology, and many other areas. For some excellent surveys of the properties of expanders, see \cite{hoory2006expander}, \cite{lubotzky2010discrete}, or \cite{lubotzky2012expander}.

Although there are many different (equivalent) ways to define expansion in a graph, we shall work with the following notion. Fix $d\in \N$, and let $\{G_n\}$ be a sequence of $d$-regular graphs such that $|V(G_n)|\to \infty$ as $n\to\infty$. Then $\{G_n\}$ is an expander sequence if there exists a constant $K$ such that, for every family of functions $f_n: V(G_n)\to \R$, 
\[ 
\frac{1}{dn} \sum_{u\sim_G v} |f_n(u)-f_n(v)|^2 \geq \frac{K}{n^2}\sum_{u, v\in V(G)} |f_n(u)-f_n(v) |^2
\]
asymptotically almost surely.

This is equivalent to the standard notion of bounding the first eigenvalue of the normalized Laplacian away from zero, by noting that
\begin{equation}\label{E:lambda1}
\lambda_1(G) = \inf_{f:V(G)\to \R}\frac{\vol{G}\sum\limits_{u\sim_Gv}(f(u)-f(v))^2}{\sum\limits_{u, v\in V(G)} (f(u)-f(v))^2d_ud_v}~,
\end{equation}
and in the case of a $d$-regular graph, $d_u=d_v=d$ and $\vol{G}=\frac{nd}{2}$ (see, for example, \cite{chung1997spectral}).

 As noted by Mendel and Naor in \cite{mendel2014expanders}, we automatically have the inequality
\begin{equation}\label{upperbd}
\frac{1}{dn} \sum_{u\sim_G v} |f(u)-f(v)|^2 \leq \frac{4}{n^2}\sum_{u, v\in V(G)} |f(u)-f(v)|^2~,
\end{equation}
for every function $f:V(G_n)\to \R$, and thus we may say that $G_n$ is an expander family if, for any sequence of functions $f_n:V(G_n)\to \R$, the average distance between function values in $G_n$ can be estimated up to a universal constant by the average distance of function values considered only along edges in $G_n$.

We shall consider here a generalization of this definition, seen previously in \cite{mendel2010towards}, \cite{mendel2013nonlinear}, \cite{mendel2014expanders}. Specifically, we can replace $\R$ and the $\ell^2$ metric above with any arbitrary metric space $(X, d_X)$. Even in this context, inequality (\ref{upperbd}) is still sure to hold (see inequality (2) in \cite{mendel2014expanders}). For a given $d$-regular graph $G$, define $\gamma(G, d_X)$ to be the infimum over all constants $\gamma$ such that for all functions $f:V(G)\to X$, the inequality
\[ \frac{\gamma}{dn} \sum_{u\sim_G v} d^2_X(f_n(u), f_n(v)) \geq \frac{1}{n^2}\sum_{u, v\in V(G)} d_X^2(f_n(u), f_n(v))\]
holds. We shall say that a family $G_n$ is an expander family with respect to a metric space $(X, d_X)$ if there exists a constant $C$ such that $\gamma(G_n, d_X)\leq C$ for all $n$.

Although we shall deal almost exclusively with regular graphs in this work, we can adapt the definition above for irregular graphs as follows. For a given graph $G$, define $\gamma(G, d_X)$ to be the infimum over all constants $\gamma$ such that for all functions $f:V(G)\to X$, the inequality
\[{\gamma}\vol{G} \sum_{u\sim_G v} d^2_X(f(u), f(v)) \geq \sum_{u, v\in V(G)} d_X^2(f(u), f(v))d_ud_v\]
holds. 

We note here that there is an equivalent notion for defining $\lambda_1(G, d_X)$. Indeed, as written, we have that $\lambda_1(G, d_\R) = \frac{1}{\gamma(G, d_\R)}$, and thus bounding $\gamma$ from above is equivalent to bounding $\lambda_1$ from below. Indeed, it may be a fruitful avenue of investigation to consider the extensions of $\lambda_1(G, d_X)$ to higher eigenvalues and applications to graph theoretic properties. However, in order to remain consistent with \cite{mendel2014expanders} and others, we will work throughout this paper with the notation $\gamma(G, d_X)$.

The nonlinear spectral gap $\gamma(G, d_X)$ has primarily been studied in the geometrical context of embeddings (see, for example, \cite{bartal2003metric, lafforgue2008renforcement, lee2007nonlinear, linial1995geometry, matouvsek1997embedding}). Indeed, one can view the spectral gap in this context as giving a bound on the distortion of an embedding of the metric space (a.k.a.\ graph) $G$ into the metric space $X$. The smaller $\gamma$ is, the more accurate such an embedding will be in preserving distances from the original graph. Much of the existing work in this topic has been on dimensionality of the space $X$, in particular when $X=\R^k$ under the $\ell^p$ metric, rather than on expansion properties of the graph $G$ (see, for example, \cite{indyk2000stable, lee2007nonlinear, matouvsek1997embedding}).

If the $G_n$ are a sequence of random graphs, we say that $G_n$ is an expander with respect to $(X, d_X)$ a.a.s.\ if there exists a constant $C$ such that $\p(\gamma(G_n, d_X)\leq C)\to 1$ as $n\to\infty$. In \cite{mendel2014expanders}, Mendel and Naor consider the question of when a classical (random) expander family is an expander with respect to an arbitrary metric space $(X, d_X)$. In particular, the authors study the expansion of a graph with respect to another graph. That is to say, the metric space $(X, d_X)$ will be a graph $H$, with $d_H(u, v)$ the standard graph-theoretic distance. In section 2.1 of \cite{mendel2014expanders}, the authors pose the following question due to Jon Kleinberg, regarding random regular graphs:

\begin{question} Does there exist a universal constant $C>0$ such that for $G\in \mathcal{G}_{n, 3}$ and $H\in \mathcal{G}_{m, 3}$ random 3-regular graphs on $n$ and $m$ vertices, respectively,
\[\lim_{m, n\to \infty} \p(\gamma(G, d_H)\leq C) = 1?\]
\end{question}

In \cite{mendel2014expanders} the authors show that such a $K$ can be
found when $m \geq n$, and provide a partial answer when $m=n$ by
considering only permutation functions. We provide here a more general
(yet still partial) set of answers in the case $m \leq n$ and having
replaced $3$ with an arbitrary regularity constant $d$, for a whole
set of growth parameters and types of functions (see the table at the
end of this section).


To do so, we introduce the following notation. For graphs $G$ and $H$, let $n=|V(G)|$ and $m=|V(H)|$. For $f:V(G)\to V(H)$, let $\gamma(G, d_H, f)$ be such that 
\[ \frac{\gamma(G, d_H, f)}{dn} \sum_{u\sim_G v} d^2_H(f(u),f(v)) = \frac{1}{n^2}\sum_{u, v\in V(G)} d_H^2(f(u),f(v)).\]

Thus, $\gamma(G, d_H)= \displaystyle\sup_{f:V(G)\to V(H)} \gamma(G, d_H, f)$. For any $\delta$, we shall say that $f$ is in class $\mathcal{F}_{G, H}(\delta)$ if, for all $i\in V(H)$, we have $|f^{-1}(i)|\leq \frac{n}{\delta}$. Note that as the expected size of $|f^{-1}(i)|$ is $\frac{n}{m}$ if $f$ is chosen randomly, this can be seen as a way to determine how ``atypical'' the function $f$ is. We shall frequently write $\mathcal{F}(\delta)$ when the graphs $G$ and $H$ are understood. Our results then fall into two categories: results bounding $\gamma(G, d_H, f)$ for all ``typical'' functions $f$, and results bounding $\gamma(G, d_H, f)$ for a predetermined but arbitrary function $f$.

\begin{remark} We will use the notation $G_n$ for a random graph $G_n
  \in \mathcal{G}_{n,d}$, respectively $H_m$ if $H_m \in
  \mathcal{G}_{m,d}$. However, if the size of the graphs is clear from
  the context, we will drop the indices and use the notations $G$ and
  $H$ instead. Similarly, when we speak about sequences of functions
  $f_n ~:~ V(G) ~\rightarrow~V(H)$, the functions will be
  deterministic and map a set of $n$ labels into a set of $m$ labels,
  for each given $n,m$.
\end{remark}

For typical functions, we have the following three theorems, depending on the relationships between $m$, $n$, and $d$.

\begin{theorem}\label{T:typicalfixed}
Let $m$ and $d$ be fixed positive integers, let $H$ be any fixed, connected $d$-regular graph on $m$ vertices, and let $G\in \mathcal{G}_{n, d}$. Then there exists a constant $C$ such that $$\p(\gamma(G, d_H)\leq C)\to 1~.$$
\end{theorem}

In \cite{mendel2014expanders}, a proof of this theorem in the case $d=3$ is mentioned and fairly briefly sketched. We extend this proof to arbitrary $d$, and include here a complete proof, relying on Bourgain's Embedding Theorem.

\begin{theorem}\label{T:typical}
For any $\eps>0$ and $d$ fixed, there exists a constant $C$ such that,
if $m\leq n$ and $m,n \to \infty$\footnote{This
  theorem may be of particular interest when $m$ grows poly-logarithmically with $n$.}, then
\[\p\left(\gamma(G, d_H, f)\leq C \hbox{ for all }f\in \mathcal{F}(m^{\frac{1}{2}+\frac{2}{d^2}+\eps})\right)\to 1,\]
where the probability is taken over pairs $(G, H)\in \mathcal{G}_{n, d}\times\mathcal{G}_{m, d}$, and the limit is taken as $n\to\infty$.
\end{theorem}

The proof of this theorem is an adaptation of techniques used in \cite{mendel2014expanders} to prove a very special case. A further adaptation yields the following result when $d$ is not fixed, but tends to infinity with $n$ as well.

\begin{theorem}\label{T:typicaldinf}
Let $G\in\mathcal{G}_{n, d}$, $H\in\mathcal{G}_{m, d}$, where $m, n$, and
$d$ all tend to $\infty$ in such a way that $m \leq n$ and $d =
o(m^{1/2})$. Let $\eps>0$ be a constant, and let $\{\delta_m\}$ be a
sequence of numbers such that 
 \[\delta_m\gg
 \frac{2\sqrt{2e}}{d}m^{\frac{1}{2}+\frac{4}{d+1}+\epsilon}~
\]
%
then $\p(\gamma(G, d_H, f)\leq C \hbox{ for all } f\in \mathcal{F}(\delta_m))\to 1$ as $n\to \infty$. 
\end{theorem}

\begin{remark} From the statement of the Theorem \ref{T:typicaldinf}
  we can see that, essentially, we must have that $\delta_{m} =
  \omega(m^{1/2+\epsilon}/d)$, so by necessity as $d = o(m^{1/2})$, we have that $\delta_m$
  grows to $\infty$. As should be clear from the proof of this
  theorem, any argument that will extend Theorem \ref{randomdiam} to a case when
  $d \gg m^{1/2+\epsilon'}$ for some $\epsilon'$ will result in the
  possibility of taking $\delta_m = 1$ for a range of growth of $d$. 
\end{remark}



On the other hand, we may also analyze $\gamma(G, d_H, f)$ for a fixed function (or sequence of functions) $f$. In this case, we have the following result.

\begin{theorem}\label{T:atypical}
Let $1>\eps>0$, and let $n, m$ satisfy $m^2\log m\ll n^{\eps/2}$. Let $(G, H)\in \mathcal{G}_{n, d}\times \mathcal{G}_{m, d}$, with $d$ constant, and let $f_n:V(G)\to V(H)$ be any sequence of functions. Then there exists a constant $C$ such that
\[\p\left(\gamma(G, d_H, f_n)\leq C\right)\to 1.\]
\end{theorem}

To prove this theorem, we shall first establish the following result for a fixed graph $H$.

\begin{theorem}\label{T:atypicalfixed}
Let $1>\eps>0$, and let $G\in\mathcal{G}_{n, d}$. Let $H$ be any fixed, connected graph, and let $f_n :V(G)\to V(H)$ be any sequence of functions. There exists a constant $C$ such that 
\[\p\left(\gamma(G, d_H, f_n)\leq C\right)\to 1.\]
\end{theorem}

In order to help the reader navigate the five main results we have
proved here, we provide the following table cataloguing the growth
regimes and indicating which results are applicable.

\begin{center}
\begin{table}[h]
\begin{tabular}{|c|c|c|c|c|} 
\hline
Fixed & Growing & Conditions & Type of function & Result \\
parameters & parameters & & & \\
\hline
$d, m$  & $n$ &$H$ fixed & all & Theorem \ref{T:typicalfixed} \\
\hline
$d, m$& $n$& $H$ fixed& any (fixed) & Theorem \ref{T:atypicalfixed} \\
\hline
$d$ & $m, n$ & $m \leq n$ & typical & Theorem \ref{T:typical} \\
\hline
$d$&$m,n$ & $m \leq n$, & any (fixed) & Theorem \ref{T:atypical} \\
& & $m^2 \log m \ll n^{\epsilon/2}$ &  &  \\
\hline
& $d,m,n$ & $m\leq n$ &typical  & Theorem \ref{T:typicaldinf}\\
& & $d = o(m^{1/2})$ &  &  \\
\hline
\end{tabular}
\vspace{.4cm}
\caption{Summary of main results and corresponding parameter regimes.}
\end{table}
\end{center}

The remainder of this paper will be structured as follows. In Section \ref{notations}, we introduce the tools and notations we shall use to prove these theorems. In Section \ref{typicalfunction}, we adapt the technique used in \cite{mendel2014expanders} for a special case of typical functions to prove Theorems \ref{T:typicalfixed}, \ref{T:typical}, and \ref{T:typicaldinf}. In Section \ref{atypicalfunctions}, we prove Theorems \ref{T:atypicalfixed} and \ref{T:atypical}.

\section{Tools and Notation}\label{notations}

Throughout, we write $\mathcal{P}_{n, d}$ to be the probability space of $d$-regular graphs on $n$ vertices, chosen according to the configuration model (see, for example, \cite{wormald1999models}). As we shall require some details of the configuration model in our proof of Theorems \ref{T:typical}, and \ref{T:typicaldinf}, we provide them here.

Let $V=\{v_1, v_2, \dots, v_n\}$ be vertices. For each $i$, let $u_{i, 1}, u_{i, 2}, \dots, u_{i, d}$ be $d$ distinct copies of $v_i$. We choose a random $d$-regular graph on $V$ as follows. First, uniformly choose a random matching $M$ on the set $U = \{u_{i, j}\ | 1\leq i\leq n, 1\leq j\leq d\}$. We note that this can be done only if $nd$ is even; indeed, if $nd$ is odd, there is no $d$-regular graph on $n$ vertices. We then form $G$ as follows. Let $V(G)=V$, and for every edge $\{u_{i_1, j_1}, u_{i_2, j_2}\}\in M$, we add the edge $\{v_{i_1}, v_{i_2}\}$ to $E(G)$.

Although this process can certainly result in a graph $G$ containing
loops or multiple edges, it is well known (see, for example,
\cite{wormald1999models}) that the probability of such a graph being
produced by this technique is asymptotically constant (converging to
$e^{-(d^2-1)/4}$), and that when we restrict to the set of simple
graphs produced under the configuration model, the distribution is
uniform over the set of simple $d$-regular graphs. We shall, therefore, restrict ourselves to $\mathcal{G}_{n, d}$, the probability space of simple $d$-regular graphs on $n$ vertices. 

Among the properties of random regular graphs we shall require is a
bound on the eigenvalues of such a graph, which we give below. 


\begin{theorem}\label{randomeigen}
Let $G\sim \mathcal{G}_{n, d}$ be a sequence of random regular
graphs such that $d$ is either fixed or $d = o(n^{1/2})$ as $n
\rightarrow \infty$. Then, with probability $1-O(n^{-2})$, there
exists a constant $C$ such that first nontrivial eigenvalue of the
normalized Laplacian for $G$ is at least $1-\frac{C}{\sqrt{d}}$.
\end{theorem}

\begin{remark} For $d$ fixed, the above is a particular case of Friedman's
standard result \cite{friedman2003proof,  friedman1989second}, while
for $d$ growing slowly so that $d = o(n^{1/2})$ is a particular case
of the more general Lemma 18 in \cite{broderfriezesuenupfal98}.
\end{remark}


The usefulness of Theorem \ref{randomeigen} is immediately apparent in the following bound on the diameter of a random regular graph.

\begin{theorem}\label{randomdiam}
Let $G\sim \mathcal{G}_{n, d}$ be a random regular graph, where $d$ is
either fixed or it may tend to $\infty$ with $n$ in such a way that $d = o(n^{1/2})$. Then with probability tending to 1 as $n$ tends to $\infty$, we have $\diam{G} = \bigTheta{\log_d n}$.
\end{theorem}

We note that the case that $d$ is fixed is a classical result due to
Bollob\'as and de la Vega \cite{bollobas1982diameter}, in which much
tighter estimates are given. We here extend to the case that $d$ tends
to $\infty$ under the condition that $d = o(n^{1/2})$ using Theorem \ref{randomeigen}, together with the following standard bound relating the diameter of a graph with its first nontrivial eigenvalue (see, for example, \cite{chung1991constructing} or Corollary 3.2 in \cite{chung1997spectral}):
\[ \diam{G} \leq \ceil{\frac{\log(n-1)}{\log(1/(1-\lambda))}}.\]

\begin{proof}[Proof of Theorem \ref{randomdiam}]
As noted above, we need only consider the case where $d\to\infty$ with $n$. Note that for any $d$-regular graph $G$, it is a standard exercise in graph theory to verify that $\diam{G}\geq \log_{d-1} n - \frac{2}{d} = \bigOmega{\log_d n}$. On the other hand, by the above bound combined with Theorem \ref{randomeigen}, we have that with probability at least $1-n^{-1}$, there is a constant $C$ with $\lambda \geq 1-\frac{C}{\sqrt{d}}$. Hence, with probability at least $1-n^{-1}$, we have
\begin{eqnarray*}
\diam{G} &\leq&  \ceil{\frac{\log(n-1)}{\log(1/(1-\lambda))}}\\
& \leq & \ceil{\frac{\log(n-1)}{\log(\sqrt{d}/C)}}\\
& = & \ceil{\frac{2\log(n-1)}{\log d + C'}} = \bigOh{\log_d(n)}.
\end{eqnarray*}
Hence, by combining these two inequalities, the result follows.
\end{proof}

We shall also require a standard discrepancy result for graphs. Let $G$ be any graph, and for $X\subset V(G)$, define $$\vol{X} = \sum_{x\in X} \deg_G(x)~.$$ If $X, Y\subset V(G)$ are disjoint, write $e_G(X,Y)$ to be the number of edges in $E(G)$ that have one endpoint in $X$ and the other in $Y$. The following result can be found as Theorem 5.1 in \cite{chung1997spectral}.

\begin{lemma}\label{discrepancy}
Let $G$ be a graph, and let $0=\lambda_0\leq \lambda_1\leq \dots\leq \lambda_{n-1}$ be the eigenvalues of the normalized Laplacian matrix $\mathcal{L}=I-D^{-1/2}AD^{-1/2}$. Let $\overline{\lambda} = \max\{|1-\lambda_1|, |\lambda_{n-1}-1|\}$. Then for all $X, Y\subset V(G)$, we have
\[\left|e_G(X, Y) - \frac{\vol{X}\vol{Y}}{\vol{G}}\right|\leq \overline{\lambda}\sqrt{\vol{X}\vol{Y}}.\]
\end{lemma}

Note that for a regular graph, we have $\vol{X} = d|X|$, so that $\frac{\vol{X}\vol{Y}}{\vol{G}} = \frac{d|X||Y|}{n}$, the expected number of edges between $X$ and $Y$. 

Throughout, we shall use the following notation. Let $f:[n]\to[m]$. For all $i\in [m]$, write $S_i=\{v\in[n]\ | \ f(v)=i\}=f^{-1}(i)$, and write $s_i=|S_i|$. If the function $f$ is not clear from context, we will use $S_i(f)$ and $s_i(f)$ to clarify. Note that in this way, we may think of a function $f:V(G)\to V(H)$ as a partition of the vertex set of $G$ into $m$ parts, where $m=|V(H)|$.

Therefore, we may rewrite the above discrepancy result in a more appropriate formulation for our special case.

\begin{lemma}\label{discrepancytwo}
Let $G\in \mathcal{G}_{n, d}$ and $C>0$. There exists a constant $\overline{\lambda}>0$ such that, with probability at least $1-Cn^{-2}$, 
\[\left|e_G(S_i, S_j) - \frac{ds_is_j}{n}\right|\leq d\overline{\lambda}\sqrt{s_is_j}\]
for all $i, j$ and all functions $f:[n]\to[m]$.
\end{lemma}

In addition to these discrepancy results, we shall also require the following concentration inequality from \cite{wormald1999models}. Suppose $G$ is a graph obtained from the configuration model. We say $G'$ is obtained from $G$ by switching two edges if there exists $\{u, v\}, \{x, y\}\in E(G)$ such that $E(G')=E(G)\ba\{\{u, v\}, \{x, y\}\}\cup\{\{u, x\}, \{v, y\}\}$. This switching occurs in the matching obtained in configuration.

\begin{theorem}[\cite{wormald1999models}, Theorem 2.19] \label{concentration}Let $X$ be a random variable defined on $\mathcal{P}_{n, d}$, such that if $G, G'\in \mathcal{P}_{n, d}$ with $G'$ obtained from $G$ by switching two edges, then $|X(G)-X(G')|\leq c$. Then for all $\lambda>0$, $\p(|X-\e{X}|\geq\lambda)\leq 2\exp\left(\frac{-\lambda^2}{dnc^2}\right)$.
\end{theorem}

Finally, to prove Theorem \ref{T:typicalfixed}, we shall use Bourgain's Embedding Theorem \cite{bourgain1985lipschitz}. Although this theorem takes many forms, the specific version we shall use is as follows (see, for example, \cite{Magen2006}).

\begin{theorem}\label{T:Bourgain} 
There exist constants $c, C>0$ such that, for all finite metric spaces $X$, there exists a function $g:X\to \R^K$, where $K = \bigTheta{\log^2|X|}$ such that, for all $x, y\in X$,
\[ (c\log |X|) d_X(x, y)\leq \|g(x)-g(y)\|_2\leq (C\log^2|X| )d_X(x, y).\]
\end{theorem}

We note that the constants $c, C$ are independent of the metric space $X$. This result can be useful for analyzing expansion in regular graphs, in particular because we have the following theorem from \cite{hoory2006expander} regarding embeddings of expanders into Euclidean spaces of arbitrary dimension. 

\begin{lemma}[Theorem 13.9 in \cite{hoory2006expander}]\label{T:Hilbert}
Let $G$ be a $k$-regular graph with first non-trivial Laplacian eigenvalue $\lambda_1$. Then for every $s>0$ and every function $f:V(G)\to \R^s$, we have
\[  \frac{\lambda_1}{n^2}\sum_{u, v\in V(G)}\|f(u)-f(v)\|_2^2\leq \frac{1}{nd}\sum_{u\sim v} \|f(u)-f(v)\|_2^2.\]
\end{lemma}


\section{Typical Functions}\label{typicalfunction}
In this section we prove Theorems \ref{T:typicalfixed},
\ref{T:typical}, and \ref{T:typicaldinf}. We begin with the proof of
Theorem \ref{T:typicalfixed}, which shall rely on Theorem \ref{T:Bourgain}.

\begin{proof}[Proof of Theorem \ref{T:typicalfixed}]
Let $H$ be a fixed $d$-regular graph with $|V(H)|=m$. Let $g:V(H)\to \R^K$ be the embedding guaranteed by Theorem \ref{T:Bourgain}.

Let $G\sim \mathcal{G}_{n, d}$, and let $f:V(G)\to V(H)$ be any
function. Define $F =g\circ f:V(G)\to \R^K$. By Theorem
\ref{randomeigen} and Lemma \ref{T:Hilbert} we have that
\[\frac{\frac{1}{nd}\sum_{u\sim_Gv}\|F(u)-F(v)\|^2}{\frac{1}{n^2}\sum_{u,
    v}\|F(u)-F(v)\|^2}\geq \alpha~,\]
where $\alpha = 1 - \frac{C}{\sqrt{d}}>0$.

Now, let us consider the function $f$. We have
\begin{eqnarray*}
\frac{1}{n^2}\sum_{u, v\in V(G)}d_H^2(f(u), f(v)) & \leq & \frac{1}{n^2} \sum_{u, v\in V(G)} \left(\frac{\|g(f(u))-g(f(v))\|_2}{c\log m}\right)^2\\
& \leq & \frac{1}{c^2\log^2m}\left(\frac{1}{n^2} \sum_{u, v\in V(G)} \|F(u)-F(v)\|_2^2   \right)\\
& \leq & \frac{1}{\alpha c^2\log^2m}\left(\frac{1}{nd}\sum_{u\sim_Gv}\|F(u)-F(v)\|_2^2\right)\\
& = & \frac{1}{\alpha c^2\log^2m} \left(\frac{1}{nd}\sum_{u\sim_Gv}\|g(f(u))-g(f(v))\|_2^2\right)\\
& \leq & \frac{1}{\alpha c^2\log^2m} \left(\frac{1}{nd}\sum_{u\sim_Gv}(C\log^2m d_H(f(u), f(v)))^2\right)\\
& = & \frac{C^2\log^2m}{\alpha c^2} \left(\frac{1}{nd}\sum_{u\sim_Gv}d_H(f(u), f(v))^2\right).
\end{eqnarray*}

Hence, for all $f:V(G)\to V(H)$, we have $\gamma(G, d_H, f)\leq \frac{C^2\log^2m}{\alpha c^2}$, which is constant with respect to $n$. Hence, $\gamma(G, d_H)$ is bounded above by a constant for all $n$, yielding the result.

\end{proof}



We now turn our attention to Theorems \ref{T:typical} and \ref{T:typicaldinf}. To prove these two theorems, we adapt the technique used in Proposition 8.1 of \cite{mendel2014expanders}. As the proof technique for both is similar, we shall combine these into one proof.

\begin{proof}[Proof of Theorems \ref{T:typical}, \ref{T:typicaldinf}]
Let $G\in \mathcal{G}_{n, d}$ and $H\in\mathcal{G}_{m, d}$. We wish to show that there exists a constant $C$ such that, with high probability, for any $f\in\mathcal{F}(\delta_m)$, we have
\begin{equation}\label{defprop}\frac{1}{n^2}\sum_{u, v\in V(G)} d_H(f(u), f(v)))^2 \geq \frac{C}{dn}\sum_{\{u, v\}\in E(G)} d_H(f(u), f(v))^2.\end{equation}

Note that the left hand side a.a.s.\ satisfies
\[ \frac{1}{n^2}\sum_{u, v\in V(G)} d_H(f(u), f(v)))^2 \leq \frac{1}{n^2}(n^2\diam H^2)=\diam H^2,\]
it suffices to prove that a.a.s.\ the right hand side satisfies
\begin{equation}\label{goal} \frac{1}{n}\sum_{\{u, v\}\in E(G)}
  d_H(f(u), f(v))^2 \gtrsim \diam H^2.\end{equation}

We first fix a graph $H$, and let $D=\diam H$. Fix a function $f:V(G)\to V(H)$ with $f\in \mathcal{F}(\delta_m)$. For now, we leave $\delta_m$ as a variable, as its value will differ depending on the case we consider. Let $\alpha>0$; the precise value of $\alpha$ will be given further on. Define
\[ N_H=\left\{(u, v)\in {[n]\choose 2}\ \vert\ d_H(f(u), f(v))\leq \alpha D\right\}.\]
Note that the number of pairs $(i, j)\in{[m]\choose 2}$ such that $d_H(i, j)\leq \alpha D$ is at most $\frac{m(1-d^{\alpha D+1})}{2-2d}$, as for a given vertex $v$ there are no more than $1+d+d^2+\dots+d^k=\frac{1-d^{k+1}}{1-d}$ vertices at distance at most $k$ from $v$. 

Therefore, as the number of vertices in $[n]$ that map to such a pair under $f$ is at most $n^2/\delta_m^2$, we have
\[|N_H|\leq \frac{n^2}{\delta_m^2}\frac{m(1-d^{\alpha D+1})}{2-2d}.\]

If it were the case that $|E_G\cap N_H|\leq \beta dn$, for some $\beta<\frac{1}{2}$, we would have
\[\frac{1}{n}\sum_{\{u, v\}\in E(G)} d_H(f(u), f(v))^2\geq \frac{1}{n}(1-\beta)dn(\alpha D)^2 \gtrsim D^2,\]
as desired. We will in fact show that for all fixed $H\in \mathcal{G}_{m, d}$, $|E_G\cap N_H|\leq \beta dn$ a.a.s.\ for an appropriate choice of $\beta$, which proves the desired result.

Let $P=[n]\times [d]$. Order the pairs in $P$ arbitrarily, and choose
a random matching $M=\{e_1, e_2, \dots, e_{\frac{nd}{2}}\}$ from $P^2$
sequentially, as in the configurational model. Define $Y_i$ to be a
Bernoulli random variable that takes the value $1$ if $e_i\in N_H$ and
$0$ otherwise. As we choose the $e_i$ sequentially, we will have
$\p(Y_i=1):=p_i \in [0,1]$, such that each $p_i$ depends upon the choice of $e_1, e_2, \dots, e_{i-1}$. However, we have the simple bound 
\[ p_i\leq \frac{|N_H|}{{nd-2(i-1)\choose 2}},\]
as upon the $i^{\textrm{th}}$ choice there are ${nd-2(i-1)\choose 2}$ edges available for $M$. 

Let $1/2>\eta>1/2-\beta$. For $i\leq \eta nd$, we have
\begin{eqnarray*}
p_i & \leq & \frac{|N_H|}{{nd-2(i-1)\choose 2}}\\
& \lesssim & \frac{2\frac{n^2}{\delta_m^2}\frac{m(1-d^{\alpha D+1})}{2-2d}}{(nd-2i)^2}\\
& \leq & \frac{md^{\alpha D}}{\delta_m^2d(d-1)(1-2\eta)^2}=:p,
\end{eqnarray*}
where $p$ is now independent of $i$. Note
that by Theorem \ref{randomdiam}, in our two cases of interest, that
is, if $d$ is either fixed or tending to $\infty$ with $n$, we have
that $D = \bigTheta{\log_d(m)}$. Let $r$ be a constant with $D\leq
r\log_dm$. Then we obtain, with high probability, that
\begin{eqnarray*}
p & = & \frac{m^{1+\alpha r}}{\delta_m^2d(d-1)(1-2\eta)^2}.
\end{eqnarray*} 

We note that some restrictions on $\delta_m$ and $\eta$ will be necessary in order that $p<1$, this shall be further discussed below.

We now produce the standard coupling $X_i$, where 
\[ X_i = \left\{\begin{array}{ll} 1 & \hbox{if }Y_i=1\\ 1 & \hbox{if }Y_i=0 \hbox{ with probability }\frac{p-p_i}{1-p_i}\\ 0 & \hbox{otherwise}\end{array}\right.,\]
so that $\p(X_i=1)=p$, $X_i\geq Y_i$, and the $X_i$ are independent. Thus we have

\begin{eqnarray*}
\p(|E_G\cap N_H|>\beta dn)  & = & \p\left(\sum_{i=1}^\frac{nd}{2} Y_i>\beta dn\right)\\
& \leq & \p\left(\sum_{i=1}^{\eta nd} Y_i>\beta dn - \left(\frac{nd}{2}-\eta nd\right)\right)\\
& \leq & \p\left(\sum_{i=1}^{\eta nd} X_i>nd\left(\beta -\frac{1}{2}+\eta \right)           \right)\\
& \leq & \left(\frac{e^{\frac{d}{p}(\beta -\frac{1}{2}+\eta )-1}}{\left(\frac{d}{p}(\beta -\frac{1}{2}+\eta )\right)^{\frac{d}{p}(\beta -\frac{1}{2}+\eta )}}\right)^{\eta ndp}\\
& = & \left(\frac{e^{d(\beta -\frac{1}{2}+\eta )-p}}{\left(\frac{d}{p}(\beta -\frac{1}{2}+\eta )\right)^{d(\beta -\frac{1}{2}+\eta )}}\right)^{\eta nd}~.
\end{eqnarray*}

In order to use a union-bound argument, we need to approximate the number of functions $f$ which satisfy the condition; a simple-minded (yet fairly accurate) estimate will be the total number of functions from $[n]$ to $[m]$: $m^n$. Hence, it suffices to show that
\begin{equation}\label{finalprob}  \left(\frac{e^{d(\beta -\frac{1}{2}+\eta )-p}}{\left(\frac{d}{p}(\beta -\frac{1}{2}+\eta )\right)^{d(\beta -\frac{1}{2}+\eta )}}\right)^{\eta d}\ll \frac{1}{m}~.\end{equation}

To simplify matters, denote by $c=d(\beta-\frac{1}{2}+\eta)$. Then
\eqref{finalprob} becomes
\[\left(\frac{e^{c-p}}{\left(\frac{c}{p}\right)^{c}}\right)^{\eta d}\ll \frac{1}{m}~.\]

Note that $ \left(\frac{e^{c-p}}{(c/p)^c}\right)\leq
\left(\frac{pe}{c}\right)^{c}$, so it suffices to verify that
$\frac{pe}{c}\ll m^{-1/(c\eta d)}$, i.e., 
\begin{eqnarray} \label{lessthanone}
p & \ll & \frac{c}{e} m^{-1/(c\eta d)}~. 
\end{eqnarray}


It is at this point that the proofs diverge. We first consider two cases, according as whether $d$ is fixed or $d\to \infty$ with $n$.

{\bf Case 1.} $d$ is fixed. Then $d$, and hence $c$, are both constant with respect to $n$, and thus we obtain the necessary condition
\[ \frac{m^{1+\alpha r}}{\delta_m^2}\ll  m^{-1/(c\eta d)}, \hbox{ that
  is, } \delta_m \gg m^{\frac{1}{2}+\frac{1}{2c\eta d}+\frac{\alpha
    r}{2}}\] Also note that \eqref{lessthanone} automatically implies
$p<1$ for $m$ large enough.

Given the constraints on the parameters $\alpha, \beta, \eta,$ and $d$ ($d\geq 3$), we can conclude that the exponent in the right-hand-side above can be made arbitrarily close to $\frac{1}{2}+\frac{2}{d^2}$ by appropriate choices of $\alpha, \beta$ and $\eta$, as $m\to \infty$. Specifically, we may choose
\[\alpha\leq\frac{\eps}{2r}, \quad \beta = \frac{1}{2}-\delta, \quad \eta = \frac{1}{2}-\delta,\]
where $\delta$ is small enough to satisfy to satisfy $\frac{1}{1-6\delta+8\delta^2}\leq 1+\frac{\eps d^2}{8}$, to obtain the condition $\delta_m\gg m^{\frac{1}{2}+\frac{2}{d^2}+\frac{\eps}{2}}$. Note that by definition this will be satisfied by $\delta_m$.

Therefore, for any fixed $H\in \mathcal{G}_{m, d}$, we have that $G$ satisfies (\ref{goal})
a.a.s.. As this holds for all $H\in \mathcal{G}_{m, d}$, the result
follows immediately, and Theorem \ref{T:typical} is proved.

{\bf Case 2.} $d\to \infty$, so we are in the situation of Theorem
\ref{T:typicaldinf}. 

In this case, we proceed a bit more delicately. We pick $\delta$ so
that $d (1- 4 \delta) = 1$, i.e., $\delta = \frac{1}{4}(1 - \frac{1}{d})$, and $\eta =
\beta = \frac{1}{2} - \delta$. As before, we have $c = d(\frac{1}{2}-2\delta)$, and thus by rewriting \eqref{lessthanone}, it is sufficient to show
\[p=\frac{m^{1+\alpha r}}{\delta_m^2d(d-1)4\delta^2}\ll
\frac{d(1-4\delta)}{2e}m^{-4/d^2(1-2\delta)(1-4\delta)} = \frac{1}{2e}
m^{-8/(d+1)}.\]

Pick now $d$ sufficiently large and $\alpha$
sufficiently small so that $m^{\alpha r/2}
\frac{1}{\sqrt{(d-1)(1-4\delta)}}< m^{\epsilon}$ for some
$\epsilon>0$. Then under the condition 
\[\delta_m\gg
\frac{2\sqrt{2e}}{d}m^{\frac{1}{2}+\frac{4}{d+1}+\epsilon},\]
we have
\begin{eqnarray*}
p = \frac{m^{1+\alpha r}}{\delta_m^2d(d-1)4\delta^2} & \ll & \frac{m^{\alpha r-2\eps-8/(d+1)}}{\frac{8e}{d}(d-1)4\delta^2}\\
& < & \frac{d(d-1)(1-4\delta)m^{-8/(d+1)}}{32e(d-1)\delta^2}\\
& = & \frac{m^{-8/(d+1)}}{2e(1-\frac{1}{d})^2} < \frac{m^{-8/(d+1)}}{2e},
\end{eqnarray*}
and hence \eqref{lessthanone} is satisfied.

We note that the only other condition, that $p<1$, is also satisfied here, as the choices of $\delta, \eta, \beta$ ensure that the right hand side of \eqref{lessthanone} is less than 1. Thus, under this condition on $\delta_m$, the result of Theorem \ref{T:typicaldinf} holds.

\end{proof}


\section{Fixed Functions}\label{atypicalfunctions}

In this section we shall prove Theorems \ref{T:atypical} and
\ref{T:atypicalfixed} using discrepancy bounds as in Lemma
\ref{discrepancytwo}. Let $G\in\mathcal{G}_{n, d}$, and let $H$ be a
graph on $m$ vertices. We assume, throughout, that $m\leq n$ and $d$
is constant. Let $f:V(G)\to V(H)$ be a function, with $S_i, s_i$ as defined in Section \ref{notations}.

Now, given $i, j\in[m]$, let $X_{i, j}(G)= e_G(S_i, S_j)$. Note that if $G'$ is obtained from $G$ by switching two edges in the configuration model as described in Section \ref{notations}, then $|X_{i, j}(G)-X_{i, j}(G')|\leq 2$. Therefore, we are in a position to apply the concentration result in Theorem \ref{concentration}. Note that $\e{X_{i, j}} = \frac{ds_is_j}{n}$.

\begin{lemma}\label{errorbound}
Let $1>\eps>0$. With probability at least $1-  \frac{2mn^{\frac{\eps}{2}}}{\sqrt c} \exp\left(-\frac{d}{4}n^{1-\eps}\right) =1-\lilOh{1}$, every pair $i, j\in[m]$ with $s_is_j\geq cn^{2-\eps}$ has \[ \frac{ds_is_j}{n}-dn^{1-\eps/2} \leq e_G(S_i, S_j)\leq \frac{ds_is_j}{n}+dn^{1-\eps/2}, \] so that $e_G(S_i, S_j) = \frac{ds_is_j}{n}(1-\lilOh{1})$.
\end{lemma}

\begin{proof}
Suppose $s_is_j\geq cn^{2-\eps}$. Then by Theorem \ref{concentration}, we have
\[\p(|X_{i, j}- \e{X_{i, j}}|\geq dn^{1-\eps/2})\leq 2\exp\left(-\frac{(dn^{1-\eps/2})^2}{4dn}\right) = 2\exp\left(-\frac{d}{4}n^{1-\eps}\right),\]
and thus, with probability at least $1-2\exp(-dn^{1-\eps}/4)$, we have $X_{i, j}\geq \frac{ds_is_j}{n}-dn^{1-\eps/2} =\e{e_G(S_i, S_j)}(1-\lilOh{1})$, and likewise for the upper bound.

Moreover, in order that $s_is_j\geq cn^{2-\eps}$, we must have that at least one of $s_i$ or $s_j$ is at least $\sqrt{c}n^{1-\frac{\eps}{2}}$. Note that there are at most $\frac{n}{\sqrt{c}n^{1-\frac{\eps}{2}}} = \frac{1}{\sqrt c} n^{\frac{\eps}{2}}$ such sets, and thus there are at most $m \frac{1}{\sqrt c} n^{\frac{\eps}{2}}$ pairs $i, j\in[m]$ that satisfy the inequality. Therefore, by the union bound, we have simultaneous concentration for all pairs satisfying the inequality with probability at least 
\[1-  \frac{2mn^{\frac{\eps}{2}}}{\sqrt c} \exp\left(-\frac{d}{4}n^{1-\eps}\right).\]
As $m\leq n$, this probability is $1-\lilOh{1}$, as desired. 
\end{proof}

We note that in this theorem the choice of $H$ is irrelevant, and all probabilities depend only on the structure of $G$.

We are now in a position to prove Theorem \ref{T:atypicalfixed}.
\begin{proof}[Proof of Theorem \ref{T:atypicalfixed}]

Let $G, H, f, S_i, s_i$ be as in the above argument. For any pair $i, j\in [m]$, let $\mathcal{E}_{i, j}=ds_is_j/n-e_G(S_i, S_j)$, the error in the approximation to $e_G(S_i, S_j)$ by its expectation. Note that Lemma \ref{errorbound} states that $|\mathcal{E}_{i, j}|= \lilOh{ds_is_j/n}$ whenever $s_is_j\geq cn^{2-\eps}$, with high probability. Moreover, note that for all $i, j$, we have $|\mathcal{E}_{i, j}|\leq \overline{\lambda} d\sqrt{s_is_j}$, by Lemma \ref{discrepancytwo}. Therefore, we have
\begin{eqnarray}
\frac{1}{n}\!\!\sum_{\{u, v\}\in E(G)} \!\!\! d_H^2(f(u), f(v)) 
\!\!& = & \frac{1}{n}\sum_{i, j\in [m]}e_G(S_i, S_j)d_H^2(i, j)\nonumber\\
\!\!& = & \frac{1}{n}\sum_{i, j\in[m]} \left(\frac{ds_is_j}{n} - \mathcal{E}_{i, j}\right) d_H^2(i, j)\nonumber\\
\!\!& \geq & \frac{1}{n}\sum_{i, j\in[m]}\frac{ds_is_j}{n}(1-\lilOh{1})d_H^2(i, j) - \frac{1}{n} \sum_{\substack{i, j\in[m]\nonumber\\ s_is_j<cn^{2-\eps}}} \mathcal{E}_{i, j}d_H^2(i, j)\label{conca}\\
\!\!& \geq & \frac{d}{n^2}(1-\lilOh{1})\sum_{i, j\in[m]}s_is_jd_H^2(i, j)- \frac{1}{n} \sum_{\substack{i, j\in[m]\\s_is_j<cn^{2-\eps}}}\overline{\lambda} d\sqrt{s_is_j}d_H^2(i, j) \label{concb}\\
\!\!& \geq &   \frac{d}{n^2}(1-\lilOh{1})\sum_{u, v\in[n]}d_H^2(f(u), f(v)) - \frac{\overline{\lambda}d\sqrt c}{n}\!\!\!\sum_{\substack{i, j\in[m]\\s_is_j<cn^{2-\eps}}}\!\!\!n^{1-\eps/2}d_H^2(i, j)\nonumber\\
\!\!& \geq &  \frac{d}{n^2}(1-\lilOh{1})\sum_{u, v\in[n]}d_H^2(f(u), f(v)) - \frac{\overline{\lambda}d\sqrt c}{n^{\eps/2}}\sum_{i, j\in[m]}d_H^2(i, j)\nonumber\\
\!\!& \geq &  \frac{d}{n^2}(1-\lilOh{1})\sum_{u, v\in[n]}d_H^2(f(u), f(v))- \frac{\overline{\lambda}d\sqrt c}{n^{\eps/2}}m^2\diam{H}\nonumber.
\end{eqnarray}
Note that line (\ref{conca}) follows from Lemma \ref{errorbound} with
probability at least $1-  \frac{2mn^{\frac{\eps}{2}}}{\sqrt c}
\exp\left(-\frac{d}{4}n^{1-\eps}\right)$ and that every
other step in the computation is deterministic. Thus, the result holds.
\end{proof}

To extend this result to Theorem \ref{T:atypical}, we need Lemma
\ref{randomeigen}
 and that with high probability, the diameter of $H$ is on the order of $\log_d m$ (see Theorem \ref{randomdiam}). Therefore, Theorem \ref{T:atypical} follows immediately by an application of the union bound.

\section{Conclusions and Open Questions}

Although we have obtained several partial results, Jon Kleinberg's
original question stated in the introduction remains open. As noted in
the introduction, and demonstrated in Theorems \ref{T:typicalfixed},
\ref{T:typical}, and \ref{T:typicaldinf}, the behavior of typical
functions and atypical functions can be somewhat different. Based on
calculations of extreme cases of atypical functions, such as functions
in which almost all vertices of $G$ are sent to the same image point,
or in which there are only two nonempty image points in $H$, it seems
that in these cases, both sides of inequality (\ref{defprop}) are
asymptotically 0, whereas for typical functions, both sides of
inequality (\ref{defprop}) are asymptotic to $\diam{H}^2$. 

It is therefore the belief of these authors that the answer to Kleinberg's question will be in the affirmative. The difficulty in resolving the question in its entirety seems to be handling the cases in between, where the function is not typical in the sense of Theorem \ref{T:typical}, but is also not as imbalanced as the cases described above.

In addition, many open problems surround the use of $\gamma(G, d_H)$ (or, equivalently, $\lambda_1(G, d_H)$) for arbitrary graphs $G$ and $H$. Preliminary work of the second author and Christopher Williamson suggests that, considered as $\lambda_1(G, d_H)$, one can glean substantial structural information about $G$ from the constant, as one would from the standard first eigenvalue.

\bibliographystyle{abbrv}  
 \bibliography{bib_items}

\end{document}